\documentclass[11pt]{amsart}
\usepackage{geometry}          
\usepackage{graphicx}
\usepackage{amssymb}
\usepackage{epstopdf}
\title{A Canonical Partition of the Primes of Logic Functions}
\author{Sidnie Feit}

\begin{document}

\theoremstyle{plain}  \newtheorem{theorem}{Theorem}[section]
\theoremstyle{plain}  \newtheorem{lemma}{Lemma}[section]
\theoremstyle{plain}  \newtheorem{corollary}{Corollary}[section]
\theoremstyle{definition}  \newtheorem{definition}{Definition}[section]

\maketitle
\begin{center}
Copyright 2012, 2013, 2014 Sidnie Feit
\end{center}


\section{Introduction}
This paper deals with Boolean functions, called \emph{logic functions} in circuit and VLSI theory.\footnote{The book [\ref{Hachtel}] by Hachtel and Somenzi provides a clear presentation of the role of Logic (i.e., Boolean) functions in VLSI theory.} 
A Boolean function in $n$ binary variables, 
\[f(x_1,\dots, x_n)\rightarrow \{0,1\},~ x_i \in \{0,1\}, ~i=1,\dots,n\] 
assigns a value of $1$ (TRUE) or $0$ (FALSE ) to each of the $2^n$ n-tuples that make up its domain. A Boolean function can be represented by a formula built using AND, OR and NOT operations. In this paper, `$*$' 
represents AND, `$+$' represents OR,  and $x'$ represents the negation NOT $x$ (also called the complement of $x$). 

An AND of variables (some of which may be negated) is called a \emph{product}, and an OR of products is called a \emph{sum-of-products}. Every Boolean function can in fact be represented by a sum-of-products  formula, such as
\begin{equation} 
\label{SampleSOP}
f = x_2*x_4'*x_8*x_9~+~ x_1*x_5*x_8'*x_{10}*x_{11}~+~x_1*x_6'*x_{10}'*x_{11} 
\end{equation}
A sum-of-products representation of a Boolean function is far from unique. Any Boolean function  can be represented by an enormous number of  distinct, logically equivalent sums-of-products. 

\begin{definition}
A product $X$ is an \emph{implicant} for Boolean function $f$ if the assignment of variable values that make product $X=1$ also makes $f=1$.  
\end{definition}
For example, if 
\begin{equation}
\label{Xterm}
X= x_2*x_4'*x_8*x_9*x_{12} 
\end{equation}
$X=1$ if and only if $x_4=0$ and the rest of the variables equal $1$. 
Substituting these values into 
the sum-of-products in \eqref{SampleSOP} results in $f=1$. 
Hence $X$ is an implicant for $f$. 
\begin{definition}
A prime is an implicant that no longer is an implicant if any factor is removed. 
\end{definition}
$X$ in \eqref{Xterm} is not prime, because 
if we remove factor $x_{12}$, then $Y= x_2*x_4'*x_8*x_9$ is an implicant.
Note that the set of factors of any implicant must include the set of factors of a prime implicant. In this example, $Y$ is a prime.

The set consisting of all of the primes for function $f$ is called its \emph{complete set of primes}. 
The sum of all of the primes in the complete set for $f$ is a canonical (but usually extremely inefficient)
representation of $f$. 

There is a classic problem associated with the primes of  a function.
\begin{quote}\emph{Given a Boolean function $f(x_1,\dots,x_n)$
and a cost function $C$  such that
\begin{itemize}
\item[] $C$ assigns a positive cost to each Boolean product,
\item[] $C$ is additive, that is, given a set of Boolean products $X_1,\dots,X_n$,
\item [] $\qquad Cost(X_1,\dots, X_n)=Cost(X_1) + \dots + Cost(X_n),$
\end{itemize}
find a minimum-cost set of primes whose sum is equal to $f$.
}\end{quote}

This problem originated in the world of logic circuits, and initially was aimed at finding a minimum-cost \emph{sum of products}
equal to a given logic (Boolean) function. Quine[\ref{Quine}] 
proved that a minimum-cost sum-of-products for a Boolean function $f$ must consist of a sum of primes if any definition of cost is used in which the addition of a single factor to a formula increases the cost. 

Here, we look for minimum-cost sums-of-primes for somewhat more general cost functions.
All of the cost functions in this paper are assumed to be positive and additive. 
A set of primes whose sum is equivalent to $f$ will be called a \emph{basis} for $f$. 
The cost of a basis is equal to the sum of the costs of the members of the basis.
A solution to the classic problem is given by a basis whose cost is minimal  
with respect to the given cost function.

This paper presents algorithms that partition the complete set of primes for $f$
into $N+2$ canonical disjoint subsets such that for \emph{any} positive additive 
cost function $C$,
solving the minimization problem consists of solving
$N$ separate minimization problems. 
The algorithms calculate the following:
\begin{enumerate}
\item $Essential~Primes$, which \emph{must} be part of \emph{any} basis for $f$,
\item $Unnecessary~Primes$ that \emph{cannot} be part of a minimum-cost 
basis for $f$ for any positive additive cost function,  
\item $N$ unique disjoint sets of primes, $PS_1, \dots, ~PS_N$ associated with
`covering' tables $TS_1, \dots, ~TS_N$ such that a minimum-cost basis for a positive, additive cost function $C$ consists of 
\begin{itemize}
\item[] $\{Essential~ Primes\} \cup QS_1(C) \cup \dots \cup  QS_N(C)$
\end{itemize}
\emph{where $QS_i (C) \subset PS_i$  and  $QS_i$  is a minimum-cost `cover' for $PS_i$,}\\
\emph{as determined by $C$ and $TS_i$.}
\end{enumerate}
Covering relationships are determined by an operation, $Cascade(RS,TS)\longrightarrow T$, where $RS$ is a set of products and $TS$ is the table of covering relationships associated with a set of primes $PS$. $RS$ covers $PS$ if and only if  $Cascade(RS, TS)$ is an empty table.
Thus, $QS_i \subset PS_i$ is a minimum-cost set such that $Cascade(QS_i, TS_i)$ is empty.

Tables and covering relationships in this paper are 
somewhat unusual due to the iterative nature of the algorithms. 
$QS_{i+1}$ covers $PS_{i+1}$ 
\emph{relative to the fact that} all products already covered by the primsets $PS_1,\dots PS_i$ 
have been removed from table $TS_{i+1}$ (and from the problem).
  
The paper also contains results that solve or simplify the problem of finding 
$QS_i$ when $TS_i$ satisfies certain conditions.  
If a subproblem satisfies the hypotheses
of Theorem \ref{SpanBasisTheorem}, the Span Basis Theorem,  a solution
consists of one easily identified prime.
If a subproblem satisfies
the hypothesis of Theorem \ref{IndPrimeDecomp}, 
the Independent Prime Decomposition Theorem, 
the subproblem can be decomposed into smaller subproblems, 
and at least one of these is solved by 
an easily identified prime.  

The input data for the partitioning algorithms
consists of the primes for $f$, the products 
generated by consensus combinations of the primes, and a 
table $triples(f)= \{(r,i,j)\}$ that expresses the covering relationships 
between products.

A key new concept in this paper is that of Ancestor Sets, whose elements are
easily derived from the $triples(f)$ table.
The Ancestor Theorem that follows is an indispensable result in this paper.  
\begin{theorem}[Ancestor Theorem]
If $A$ is an Ancestor Set for $f$ and $C$ is a positive additive cost function, then 
every minimum-cost basis consists of the union of the Essential Primes, 
a minimum-cost cover
 $QS(C)$ for the primes $PS$ that belong to ancestor 
 set $A$, and a minimum-cost cover for the set of primes that are 
not in $A$ (and are not covered by the union of the Essential Primes and $PS$).\end{theorem}
An \emph{Independent Ancestor Set} is an Ancestor Set that does 
not contain any other Ancestor Set.  The Independent Ancestor Sets are
unique and disjoint, 
and each $PS_i$ in the partition is equal to the the set of primes in an Independent
Ancestor Set $A_i$. 

A straightforward algorithm applied to the $triples$ table identifies a
unique initial group of disjoint
Independent Ancestor Sets, say, $A_1,\dots,A_k$, with primesets $PS_1,\dots,PS_k$.  
Because of the Ancestor Theorem, every minimum-cost basis for $f$ must include
minimum-cost covers for each primeset $PS_i, i=1,\dots,k$. 

But then all products covered by the union of 
the Essential Primes and  the $PS_i$ 
can be removed from the $triples$ table and from the partitioning problem. 
The remaining Independent Ancestor Sets are found by iterating the algorithm
on the reduced $triples$ table until the $triples$ table is empty.
A minimum-cost cover for each new Independent Ancestor Set
$A_j$ will be part of a minimum-cost basis for $f$.  

If there are $K$ products and the $triples$ table has $M$ rows,   
the partitioning algorithms execute with worst-case complexity  $M*log_2(M)*K$.
When there are some Essential Primes, the number of products and
the number of rows in the table are reduced in size significantly.

The algorithms in this paper calculate  the primes, the relevant consensus 
combinations of primes,  
the Independent Ancestor Sets $A_i$, the $PS_i$, and the tables $TS_i$. 
Although the algorithmic proofs are valid, the worst-case size of the data
can be enormous. First of all,
to get a sense of the number of primes that could be associated with
a Boolean function, McMullen and Shearer [\ref{McMullen}] found the 
following upper bound and constructed a function that attained
the upper bound:
\begin{itemize}
\item[] If $m$ equals the minimum number of products in a sum-of products that
represents $f$, then the the number of primes for $f$ is  less than or equal to $2^m - 1$.
\end{itemize}
Furthermore, the task of finding a minimum-cost $QS_i(C)$ 
can be framed as an integer programming problem,
and a general integer programming problem can be NP-hard.   

However, it is possible that new algorithms may be discovered
that split up an oversized problem by extracting 
some Ancestor Sets from the problem.
Also, the tables for some Ancestor Sets satisfy hypotheses 
that make it possible to find minimum-cost covers in linear or polynomial time.

But it is not surprising that much effort has been devoted to finding
algorithms that lead to approximate solutions. 
The classic problem has been studied extensively 
using a geometric model that represents 
products as "cubes" and merges smaller cubes into larger ones. 
This model underlies 
various versions of the popular heuristic ESPRESSO algorithm [\ref{Espresso}], which  
originally was developed by Robert Brayton at IBM. ESPRESSO frequently is used to compute
 approximations to a minimum-cost basis. Sometimes an exact solution is sought, and in this case, the first step in ESPRESSO is an attempt to calculate the 
 complete set of primes for $f$.\footnote{Another approach uses 
 data structures called Binary Decision Diagrams (see [\ref{Boute}])  represent a Boolean function as a rooted, directed, acyclic graph. These diagrams have led to advances in formal verification of circuits and logic minimization.}

\section{The Consensus Identity and Consensus Operations}\label{consensusID} 
This section and the sections that follow include statements of the type 
\[formula~E = formula~F.\]
\noindent In Boolean algebra, such a statement represents 
`formula $E$ is logically equivalent to formula $F$' or 
`formula $E$ represents the same Boolean function as formula $F$.' 
This  paper follows the usual practice of using an equals sign to denote logical equivalence.

In the following,  small letters such as $x, y, z$ represent binary ($0$ or $1$) logic variables.
Capital letters, including $P, Q, R, U, V, W, X, Y, Z$ will  represent 
products whose factors are binary variables or their negations.

The standard consensus identity and the consensus operation of Boolean algebra play a major role 
in this paper. The consensus identity for binary variables  $x$, $y$, and $z$, is simple:
\[x*y + x'*z  = x*y + x'*z +  ~~ \textbf{y*z}\]
\noindent Note that the products $x*y$ and $x'*z$ have one \emph{reversal} variable --- that is, one variable, $x$, that is not negated in one product and is negated, $x'$,  in the other. The identity also 
holds when we replace $y$ and $z$ with products $Y$ and $Z$ such that the pair $Y$ and $Z$ has no reversal variable.
\[x*Y+x'*Z  = x*Y + x'*Z+ ~~\textbf{Y*Z}\]
Product $Y*Z$ is called the \emph{consensus} of $x*Y$ with $x'*Z$, written
\[Y*Z = (x*Y~ \circ~ x'*Z).\]
The consensus operation $(U  \circ V )$ for products $U$ and $V$ is defined and valid only for products $U$ and $V$ that have exactly one reversal variable.   The result  consists of the product of all of the factors in $U$ and $V$ (with duplicates removed) except for the reversal factors. If $(U \circ V)$ and $(Y \circ W)$ are valid, then
 \[U+V+Y + W= U + V + ~\textbf{(UoV)} + Y + W +~\textbf{(YoW)}\]
If $(~(U\circ V)\circ(Y \circ W)~)$ is valid, then
\[U+V+Y + W= U + V + (U \circ V) + Y + W + (Y \circ W) + \mathbf{ ((U \circ V)\circ(Y \circ W)) } \]

\begin{definition}\label{DefConsensus} 
Suppose that $\{X_1,\dots X_k\}$ is a set of distinct products. Then 
\[Consensus(X_1,\dots,X_k)\]
 equals $\{X_1, \dots , X_k\}$ and all products generated by their valid consensus combinations.
\end{definition}

The following procedure, here called \emph{Loose Iterated Consensus}, is an orderly way to generate $Consensus(X_1,\dots, X_k), k>1$. Initially, start with the list of the distinct products, $X_1,\dots,X_k$
and set $i=1$.
\begin{enumerate}
\item Increase $i$ by 1. Compute the consensus of $X_i$ with each prior product in the list. 
If a result $X$ is new, append it to the list.
\item If $X_i$ is not the last product, repeat starting with (1). 
If $X_i$ is the last product, stop.
\end{enumerate}
Then $Consensus(X_1,\dots, X_k)$ equals the set of products in the final list, 
\[X_1,\dots, X_k, X_{k+1}, \dots, X_n.\] 

The following Lemma follows from the consensus identity.
\begin{lemma}[Consensus Lemma]\label{consensusLemma} 
Given a Boolean function $f$ expressed as sum-of-products  $X_1 +\dots + X_k$,  then the sum of the products in $Consensus(X1,...,Xk)$ is logically equivalent to $f$, and every product in $Consensus(X1,...,Xk)$ is an implicant for $f$.
\end{lemma}

\begin{definition} Product $Y$ is \emph{covered} by \{ $X_1, \dots, X_k$ \} if $Y \in Consensus(X_1,\dots,X_k)$.
\end{definition} 

\subsection{Computing the Complete Set of Primes}
The classic method of finding the complete set of primes associated with a Boolean function 
that is defined by a sum-of-products formula 
is based on a theorem proved by Brown [\ref{Brown}]  and presented in more modern language in many textbooks. The wording here differs slightly from the usual language.

One of the most basic identities in Boolean Algebra is used when computing the complete set of primes. If $U$,  $V$, and $W$ are products and $V=U*W$, then 
$U + V = U + U*W = U$. A special case is $U + U = U$.

\begin{theorem}[Brown's Complete Set of Primes Theorem] A sum-of-products formula for $f$ is the sum of the complete set of primes for the function if and only if
\begin{enumerate}
\item No product's factors are a subset of  any other product's factors.
\item If the consensus of two products $Y$ and $Z$ in the sum exists,  $X=(Y \circ Z)$, then 
either $X=$ some product in the sum, or there is a product $R$ in the sum such that the factors in 
$R$ are a proper subset of the factors in $X.$ 
\end{enumerate}
\end{theorem}

The classic \emph{Iterated Consensus} algorithm for calculating the complete set of primes for a sum-of-products, is based on Brown's Theorem.
Note that the Iterated Consensus of $X_1,  \dots ,X_m$, here called \textbf{IT} $(X_1, \dots , X_m)$, differs from $Consensus(X_1,\dots ,X_m)$. When computing \textbf{IT}, if the factors of $X$ are a proper subset in any other product $Z$, product $Z$ is removed. 

\begin{definition}\label{completeSet}
Suppose that $f=X_1+\dots +X_m$ and there is no $X_i$ whose factors are a subset of another product's factors. 
The \emph{Iterated Consensus}, \textbf{IT} $(X_1, \dots ,X_m)$  
is defined as follows. Starting with $X_2$, compute the consensus of each product with each prior product. For each valid consensus result $X$,
\begin{enumerate}
\item If the factors in $X$ are a proper subset of the factors 
in any other product(s) on the list, remove the product(s). 
\item If the factors of any $X_k$ on the list are a subset of the factors of $X$, discard $X$.  
\item If no $X_k$ is a subset of $X$, append $X$ to the list.
\end{enumerate}
\end{definition}
The resulting sum satisfies the hypotheses of Brown's Theorem and hence the products are the complete set of primes for the function.\footnote
{Another method, which often is more efficient, is to calculate the primes by recursive use of Boole's Expansion Theorem. See [\ref{Hachtel}], page 139.}  
 
At every step in performing the Iterated Consensus, the sum of the products on the list remains equivalent to the Boolean function $f$ defined by the original sum-of-products. Thus the sum of the complete set of primes is equal to the original function $f$. The same set of primes will be generated from any sum-of-products that is equal to $f$. In particular, 
if the set of primes $ Q_1,\dots,Q_m$ is a basis for $f$, then
\[\mathbf{IT}(Q_1,..., Q_m)=\text{the complete set of primes for~} f.\]

\subsection{ALL, the Consensus of the Complete Set of  Primes}
Once the complete set of primes of $f$ is known, 
the relationships between the primes are hidden within 
the relationships between \emph{all} of the products generated by computing consensus 
combinations of the primes.
For example, we may have primes $P,~\{Q_i\}$ with 
\[P=(~ (Q_1 \circ Q_2) \circ (Q_8 \circ Q_{10})~)\] 
where $X=(Q_1 \circ Q_2)$ and  $Y=(Q_8 \circ Q_{10})$ are not primes. Note that if $Q_1, Q_2, Q_8, \text{and }Q_{10}$ belonged to a minimum-cost basis, then $P$ could not
belong to the minimum-cost basis since it already is covered.

To capture all relationships, after finding the primes, let 
\[ALL~=~Consensus(\text{complete set of primes}).\]
That is, $All$ is generated by performing a Loose Iterated Consensus,
where only duplicate products are discarded.\footnote{More efficient
methods for calculating ALL can be devised, but this simple method
suffices for the purpose of defining the set.} 
The next lemma follows immediately.
\begin{lemma}\label{basisALL} 
Primes $P_1,\dots,P_k$ are a basis for $f$ if and only if 
\[Consensus(P_1,\dots,P_k)=ALL.\]
\end{lemma}

As mentioned earlier, the  number of products that need to be analyzed 
and the number of relationships between 
products are diminished if some of the products are \emph{Essential}. 
Recall that product $Y$ is \emph{covered} by \{ $X_1, \dots, X_k$ \} 
if $Y \in Consensus(X_1,\dots,X_k)$.
\begin{definition} A prime $P$ is \emph{Essential} if $P$ is not covered by the other primes.
\end{definition}
All Essential Primes have to be part of any minimum-cost basis. 
Since their fixed positive cost is irrelevant to computing a minimum-cost basis,
in this paper they are called \emph{free}. 

The Essential Primes  can be calculated directly from any \emph{sum-of-primes}
 formula for $f$ using the following Theorem of Sasao [\ref{Sasao}].
\begin{theorem}
Suppose that function $f$ is expressed as a sum of primes,
\[P + V_1 + \dots V_n.\]
Then $P$ is an essential prime if and only if $P$ is not covered by
\[Consensus(P*V_1,\dots, P*V_n,~ (P\circ V_1), \dots,(P\circ V_n)~).\]
\end{theorem}

\subsection{Free Products, Non-Free Products, and Spans}
The products generated by 
computing $Consensus(Essential~Primes)$ are covered by the Essentials,
and hence also can be viewed as \emph{free} products. 
The set of all free products is denoted by $<Free>$. Thus,
\[<Free>=Consensus(Essential~Primes).\]

For any  positive additive cost function $C$, 
the Essential Primes are a minimum-cost cover for $<Free>$.
The set of free primes that are not Essential cannot belong to any minimum-cost basis
and hence belong to the set of $Unnecessary~Primes$.

\begin{definition} Every prime that is not Essential is called a non-Essential prime. 
A non-Essential prime that is not free is called a non-free prime. 
\end{definition}

If all primes are free, the problem is solved because the unique minimum-cost  basis for any positive additive cost function $C$ consists of  the Essential Primes. From now on, we assume that there are some non-free primes. The set of products of interest is
\[NONFREE~= \text{the non-free products  in}~ ALL\]
We will order and number $NONFREE$ with the non-free primes first.
It is convenient to number the non-free primes $\{P_k\}$ in order of increasing cost. (This will make it trivial to select a prime with the least cost from a set of equivalent primes.)

\begin{definition} The set of non-free primes that turn out to be \emph{Unnecessary}
will be called $Surplus$.
\end{definition}

\begin{definition}\label{spanDef} 
The \emph{Span} of a set of non-free products $S$ is defined by 
\[Span(S)~=~\text{the non-free products in } Consensus(<Free>,~S).\]
\end{definition}

\begin{definition}
Given a set of non-free products $S$, 
non-free product $X$ is \emph{covered} by $Span(S)$
if $X\in Span(S)$.
\end{definition}

\begin{definition}
A set of non-free primes $R_1,\dots, R_g$ for $f$ is a \emph{cover} for the
non-free primes of $f$ if 
\[Span(R_1,\dots,R_g)=NONFREE\]
\end{definition}

Note that if  $Span(R_1,\dots, R_g)$ covers the non-free primes of $f$, then
\[Essential~Primes,~R_1,\dots, R_g\]
is a basis for $f$. 

Clearly, any minimum-cost basis must consist of the Essential Primes and a minimum-cost cover,  $Q_1,\dots,Q_k$, for the non-free primes. 
 
 \begin{lemma} \label{primeCovered}
Let $P_1, \dots ,P_n$ be the non-free primes. Then each non-free $P_k$ 
is covered by the span of the other non-free  primes. 
\end{lemma}
\proof
If $P_k$ was not covered by the span of the other non-free primes it would have to be an Essential or other free prime, contradicting the assumption that it is non-free. 
\endproof

\section{The $triples(f)$ Table}
Consensus relationships are used to build $triples(f)$, a table that includes
all useful covering relationships 
and is the starting point of all calculations. 

The Essential Primes must be in any basis, so we assume that they have been
chosen as the initial basis members  
and hence all free products are already covered.

A table entry, $(r,i,j)$ corresponds to a consensus relationship, $X_r=(X_i\circ X_j)$.
This relationship states that if $X_i$ and $X_j$ are covered, then $X_r$ is covered.
Note that:
\begin{enumerate}
\item When $X_r$ is free, $X_r$ already is covered by the Essentials
so the relationship is irrelevant and is not included.  
\item If both $X_i$ and $X_j$ are free, then $X_r$ is free, so the relationship is not included.
\item If $X_r$ and $X_i$ are non-free and $X_j$ is free, then, since $X_j$ already is covered, 
$X_r$ is covered whenever $X_i$ is covered. This is tabulated as $(r,i,0)$.\footnote{When 
taking the consensus of a non-free with a free, by convention, the free product 
will always be assumed to occupy the $X_j$ position.} 
Thus, all free $X_j$ products will be represented by 0s in the table,
and all product indices in a table represent non-free products.
\item If $X_r, X_i$, and $X_j$ all are non-free, then the entry is tabulated as $(r,i,j)$.
\end{enumerate}

The $(r,i,0)$ entries are called $triples0$ entries
and the $(r,i,j)$ entries are called $triples3$ entries.

In addition to the $triples0$ entries generated by consensus operations,
there is a second set of valuable $triples0$ cover relationships.
By definition, non-free prime $P_k$ covers every product in $Span(P_k)$. 
When $Span(P_k)$ is maximal (see Definition \ref{maxSpan}) and closed (see Definition \ref{closedSpan}, 
then for each $X_r\in Span(P_k)$,
we append $(r,k,0)$ to $triples0$.
 
The calculations will produce many duplicate $(r,i,0)$ results. 
Efficient code can discard duplicates as they arise 
(for example, by using a sorted list and hashed list insertion).

It is convenient to refer to products by their indices. Thus, we can say that  
an $(r,i,j)$ entry states that  if $i$ and $j$ are covered then $r$ is covered.

As a rule of thumb, a greater number of $(r,i,j),~j>0)$ entries lead to slower execution, while
more $(r,i,0)$ entries lead to faster execution.
Fortunately, when there are free products, 
some $(r,i,j), j>0$ entries are \emph{useless} and can be discarded.

Specifically, if there is an $(r, i,0)$ entry, 
 $i$ covered implies $r$ covered; thus, stating `$i \text{ AND } j$ covered 
 imply $r$ covered' is redundant and any such $(r,i,j)$ entry can be discarded. 
 The same argument holds for $(r,i,j)$ if there
 is an $(r,j,0)$ entry.  An $(r,i,j)$ entry that is not useless is called \emph{useful}.

The $triples(f)$ table consists of the union of the
remaining entries in $triples0$ and $triples3$. 
The algorithm in Appendix \ref{reorder} arranges the products in 
$NONFREE$ in a favorable order and the algorithm in  
Appendix \ref{genTable} can be used to generate $triples$. 
These algorithms are just examples, 
and can be used for problems of small size. 
More efficient algorithms are required
for big problems. 

Note that by the definition of $Span$, the set of non-free primes $R_1,\dots, R_g$ 
covers the non-free primes if and only if $Span(R_1,\dots, R_g)$ covers every index
in $triples(f)$.

\subsection{Inactive Products}
An $r$-index that does not appear anywhere as an $i$ or $j$ does not cover anything, and is called \emph{inactive}. After removing the entries for inactive $r$-indices, others may become inactive.  An iteration clears out these entries, (including inactive primes) leaving only active indices. 
An inactive prime cannot be part of any minimum-cost basis, and belongs to $Surplus$.

From now on, we assume that useless and inactive entries have been 
removed from $triples(f)$. Let $triples$ denote the resulting table of useful, active products.
Once the initial $triples$ table has been defined, no further calculations of 
consensus products are needed. All needed information is in $triples$. 

\subsection{Cascades Relative to a Table}
Given a set of products $S$ and a $triples$ table $T$, the operation $Cascade(S,T)$  plays a key role in the algorithms used to discover the  $PS_i$.
The purpose of $Cascade(S, T)$ is to
\begin{itemize}
\item[] generate the list $\{S, \text{ all other products in } T \text{ covered by }S\}$,
\item[] remove all indices in $S$ or covered by $S$ from $T$.
\end{itemize}
When the identity of the table $T$ is clear from the current context,
we can write $Cascade(S)$ instead of $Cascade(S,T)$.

The cascade process is iterative. If $S$ covers additional indices $S1$
and $S\cup S1$ covers additional indices $S2$, then $S$ covers $S\cup S1 \cup S2$.
Thus, an initial cascade list grows through iteration until no more new indices are found.
To cascade set $S$ through the given table: 
\begin{enumerate}
\item Set the initial cascade list to $S$.
\item For each $s$ in the cascade list, remove all $(s,x,y)$ entries from the table.
\item For the remaining table entries, 
for each index in an $i$ or $j$ position that belongs to the cascade list, replace the index with $0$. 
\item A $(t,0,0)$ entry indicates that $t$ is covered. Append $t$ to the
cascade list.
\item Iterate for the set of new indices in the cascade list,  until no new $(t,0,0)$ index is found. 
\end{enumerate} 
Figure \ref{FigCascade} illustrates substitutions for $S=\{32\}$
into the table in the figure.
After replacing 32 with 0, indices 25, 26, 27, and 30 are added to the cascade list.  
 \paragraph{}
\begin{figure}
\begin{center}
\begin{tabular}{l l l l l}
\hline
\small
$\mathbf{(~24,  67,  ~0)}$ & $~$\\
( 25,  32,  ~0) & $\Longrightarrow$ (25, ~0, ~0) & \\
( 26,  32,  ~0) & $\Longrightarrow $ (26, ~0, ~0) &\\
( 27,  32,  ~0) & $\Longrightarrow$  (27, ~0, ~0) & \\
( 30,  32,  ~0) & $\Longrightarrow$  (30, ~0, ~0) & \\
( 30,  35,  37) &  &$ \Longrightarrow$ \emph{deleted!}\\
( 31,  25,  37) &  &$ \Longrightarrow$ (31,~0,~37) &$ \Longrightarrow$ (31,~0,~0)&\\
( 32,  27,  31) &  $\Longrightarrow$  \emph{deleted!}\\
$\mathbf{(~33,  24,  39)}$\\
( 35,  27,  ~0)&  &$ \Longrightarrow$ (35,~0,~0) \\
( 35,  30,  ~0)&  &$ \Longrightarrow$ (35,~0,~0)\\
( 37,  30,  ~0)&  &$ \Longrightarrow$ (37,~0,~0)\\
( 38,  35,  40)&  &                     &$ \Longrightarrow$ (38, ~0, 40)=$\mathbf{(~38, 40, ~0)}$\\
\hline
\end{tabular}
\end{center}
\caption{Computing Cascade(32). $(r,0,0)$ indicates that $r$ is covered.\label{FigCascade}}
\end{figure}
\normalsize
Then these new values need to be replaced with $0's$ in the table. Iterations continue until no further changes to the list occur. In this example, the final cascade list consists of 
$32$ and $25, 26, 27, 30, 31, 35,\text{ and }37$, and the
 final table (see Figure \ref{CascadeResult}) has only three entries. 
 
\begin{figure}
\begin{center}
\begin{tabular}{l}
\hline
( 24,  67,  ~0)\\
( 33,  24,  39)\\
( 38,  40, ~0)\\
\hline
\end{tabular}
\end{center}
\caption{The result after cascading 32.\label{CascadeResult}}
\end{figure}
\normalsize
 
 \subsection{Spans Relative to a Table}
 Earlier, $Span(S)$ was defined to be the set of all non-free products in 
 $Consensus(<Free>, S)$, in other words, the non-free products that
 are covered by $S$.
 From now on, spans will be calculated relative to whatever local
 $triples$ table $T$ is under study. 
 \begin{definition}\label{span2}
$Span(S, T)$, the span of a set $S$ relative to $triples$ table $T$,  
is equal to the union of $S$ with the set of products in $T$ that are covered by $S$. 
That is, the span is the cascade list that results from cascading $S$ through a copy of $T$.
When the relevant table $T$ is clear from the context, the span can be written $Span(S)$.   
\end{definition}

 The span of an individual prime often provides useful information
 when trying to calculate a minimum-cost cover $QS_i\subset PS_i$. 


\section{Ancestor Sets}
In this section and the ones that follow, all products are active and non-free. $P$, $Q$, and $R$ will represent active non-free primes. $X$, $Y$, or $Z$ can represent any active product in $NONFREE$, including primes.  

At this point, we are ready to define and compute the Ancestor Sets associated with a $triples$ table $T$. 
\begin{definition} Given $triples$ table $T$, for any triple $(r,i,0)$ $i$ is a \emph{parent} of $r$.  For $(r,i,j),~j>0$, $i$ and $j$ are \emph{parents} of $r$.
\end{definition}

\begin{definition} The Ancestor Set of non-free product $X_r$, $Anc(X_r)$, (also written $Anc(r)$)  consists of the union of  all of the parents of $r$ with the parents of parents,  parents of parents of parents, ... iterated until there no new ancestors are added.
\end{definition}

A $0/1$ bitmap can be used to represent the parents of each non-free product. 
A simple (but inefficient) way
to find $Anc(r)$ is to set the initial $A(r)$=$parents(r)$ and then let 
\[A(r)=A(r) \text{ OR } parents(each~parent).\]
Then, for each new member of $A(r)$, set 
\[A(r)=A(r) \text{ OR } parents(each~new~member)\]
until no new ancestors are added.

\begin{definition} The Ancestor Set of a non-free prime is called a \emph{Prime Ancestor Set}.
\end{definition}

Later, it will be shown that only Prime Ancestor Sets need to be calculated, and in fact,
there are shortcuts that identify a small subset of the primes
whose ancestor sets actually need to be calculated.
The algorithms for computing ancestor sets make use of the following obvious lemma.

\begin{lemma} If $Y \in Anc(X)$ then $Anc(Y) \subset Anc(X)$.
\end{lemma}

\begin{definition} An Ancestor Set is \emph{Independent} if it does not properly 
contain any other Ancestor Set.
\end{definition}
The Independent Ancestor Sets are the critical ones. Lemma \ref{indIsPrime} will show that 
every Independent Ancestor Set is a Prime Ancestor Set.

\begin{definition}
For Ancestor Set $A$, $triples(A)$ is defined to be the set of all triples $(r,i,j)$ in $T$ 
such that $r$, $i$, and $j$, if $j>0$, belong to $A$. 
\end{definition}

\begin{lemma} Every non-free product has some ancestors that are non-free primes.
\end{lemma}
\begin{proof} 
For the non-free primes, this follows directly from Lemma \ref{primeCovered}. For other non-free products, it follows from the fact that all of the non-free products were generated 
by $Consensus(<Free>, \text{ non-free } primes)$. \end{proof}

\begin{definition} Given an Ancestor Set $A$, the \emph{primeset} for $A$ 
is the set of all primes that belong to $A$. 
\end{definition}

\begin{definition} Primes $Q_1, \dots ,Q_t$ cover $A$ if 
$A$ is equal to the span of $Q_1, \dots ,Q_t$ in $triples(A)$.
\end{definition}

\begin{lemma} Suppose that $A$ is an Ancestor Set and $PS=\{P_1, \dots ,P_m\}$ 
is its primeset.  Then $PS$ covers $A$ with respect to $triples(A)$, that is,
\[A = Span( P_1, \dots ,P_m)~in~ triples(A).\]
\end{lemma}
\begin{proof}
Every non-free product $X$ is generated by consensus combinations that 
include non-free primes, (as well as free primes that are represented by
0s in the table). $P_1, \dots ,P_m$ are the only primes in $A$ and hence are the only 
non-free prime ancestors of any $X \in A$. Hence $A$ must be contained in their span.
\end{proof}

\begin{lemma}
Primes $QS=Q_1, \dots ,Q_t$ cover $A$ if and only if they cover the
primeset $PS$ of $A$.
\end{lemma}
\proof
Since by definition, the primeset $PS$ for $A$ is contained in $A$, 
if $QS$ is a cover for $A$ it must cover $PS$. Conversely, if $QS$ covers $PS$,
the span of $QS$ must contain $A$.
\endproof

\begin{definition} An Ancestor Set $A(Y)$ is closed with respect to $Y$ if $Y \in A(Y)$.
\end{definition}

\subsection{Statement and Proof of the Ancestor Theorem}
\label{secAnc}
The  \emph{Ancestor Theorem} is a crucial result for this paper. The statement is true for any 
Ancestor Set $Anc(X)$, although it will only need to be applied to Independent Ancestor Sets.

The key fact is that every minimum-cost cover for an Ancestor Set  $A$ has to be chosen from the primeset for $A$, since all other primes are irrelevant to covering $A$.

\begin{theorem}[Ancestor Theorem]\label{ancThm}
Let $A$ be an 
Ancestor Set, $PS$ be its primeset, and $RS$ be the set of primes that are not 
covered by the union of the Essential Primes with $PS$. 
Every minimum-cost basis for the primes 
consists of the Essential Primes, a minimum-cost subset of $PS$ that covers  
$PS$, and a minimum-cost subset of $RS$ that covers $RS$.
\end{theorem}
\begin{proof}
 By the definition of Ancestor Sets, every non-free prime ancestor of any $P\in PS$ belongs to $PS$.
 Let $Q_1, \dots ,Q_k$ be a minimum-cost cover for \emph{all} of the non-free primes. That is,
\[Span(Q_1, \dots ,Q_k) \text{~with respect to }T \text{~covers all of the non-free primes.}\]
 
Some minimum-cost subset of $\{Q_i\}$ covers the primeset of  $A$. Suppose that this subset 
consists of $Q_1, \dots ,Q_t$.  
If some $Q_i, ~i\leq t$  does not belong to $PS$ (and hence does not belong to $A$), it is not an 
ancestor of any product in $PS$, or in $A$. The span of the remaining $Q_j$ must cover $PS$, 
which contradicts $Q_1, \dots ,Q_t$ being a minimum-cost set. Hence $Q_1,  \dots , Q_t$ must all belong to $PS$ and must cover $PS$.  

Suppose that $Q_1, \dots ,Q_t$ covers $PS$, but is not a minimum-cost cover for $PS$. If the set 
of primes $\{R_1, \dots ,R_m\} \subset PS$, is a minimum-cost cover for $PS$, 
then $R_1, \dots , R_m$ covers every product in $PS$ including the primes $Q_1, \dots ,Q_t$, so 
$R_1, \dots ,R_m, Q_{t+1}, \dots, Q_k$
is a better cover for the non-free primes, 
contradicting the assumption that $Q_1, \dots ,Q_k$ is a minimum-cost cover. Finally, no prime that is already covered by the union of the Essential Primes with $PS$ can be part of a minimum-cost basis.
\end{proof}

At this point, the Ancestor Sets may not be disjoint. An Ancestor Set may contain or be contained in other Ancestor Sets.

\begin{corollary}
If Ancestor Set $A$ is contained in Ancestor Set $B$, then a minimum-cost cover for $B$ contains a minimum-cost cover for $A$.
\end{corollary}

\subsection{Properties of Independent Ancestor Sets}
Recall that  $Anc(X)$ is Independent if it does not \emph{properly} contain 
$Anc(Y)$ for any $Y  \in Anc(X)$.
Note that this implies that if $Anc(X)$ is Independent, $Anc(X)=Anc(Y)$ for every $Y \in Anc(X)$.

\begin{lemma}\label{indIsPrime}
Every Independent Ancestor Set is a Prime Ancestor Set.
\end{lemma}
\proof
For each prime $Q$ in the primeset of Independent Ancestor Set $A$, $A=Anc(Q)$.
\endproof

There always is at least one Independent Ancestor Set, namely an Ancestor Set of smallest size.  
The following \emph{Independence Lemma} characterizes Independent Ancestor Sets. 

\begin{lemma}[Independence Lemma.]\label{indepLemma} Suppose that $A$ is an Ancestor Set 
and $Q_1,\dots ,Q_k$ is its primeset. Then $A$ is an Independent Ancestor Set if and only if
\[Anc(Q_1)=Anc(Q_2)=\dots=Anc(Q_k)=A.\]
\end{lemma}
\begin{proof} 
Suppose that $A$ is Independent. Since each $Q_i \in A$, $Anc(Q_i) \subset A$. But $Anc(Q_i)$ cannot be properly contained in $A$ so $Anc(Q_i)=A$.

For the converse, suppose that $Anc(Q_i)=A$ for every prime $Q_i$ in the primeset of $A$. Suppose that $B$ is an Ancestor Set properly contained in $A$. 
Every Ancestor Set must contain primes. Suppose that prime $Q \in B$. Then $Anc(Q) \subset  B$. Since $B \subset A$, $Q \in A$. But by hypothesis, $Anc(Q)=A$, which contradicts $B$ properly 
contained in $A$.
\end{proof}

\begin{lemma}
If $A_m$ and $A_k$ are distinct Independent Ancestor Sets, then $A_m$ and $A_k$ are disjoint, that is, $A_m \cap A_k$ is empty .
\end{lemma}
\begin{proof} 
If $X \in A_m$ and $X \in A_k$, then $A_m=Anc(X)=A_k$.
\end{proof}
Note that the primesets of distinct independent ancestor sets $A_m$ and $A_k$ are disjoint,  and $triples(A_m)$ is disjoint from $triples(A_k)$.

\subsection{Finding the Initial Independent Ancestor Sets}
Suppose that the $triples$ table is known and has been stripped of
all inactive or useless entries.
The procedure that follows uncovers the independent ancestor sets iteratively,
one batch at a time. The steps are easy to follow, 
and prove that the iteration will succeed.\footnote{However, the algorithm in Appendix \ref{fastInd} bypasses the calculation of most of the ancestor sets. 
It computes batches of independent ancestor sets directly using a method that converges quite quickly.} 

After the first batch of Independent Ancestor Sets has been generated,
$triples$ will be replaced by a smaller table and a fresh batch of Independent
Ancestor Sets will be calculated from the smaller table.
 
 Iteration continues until the $triples$ table is empty. 
We assume that at each step in the discussion that follows, $triples$ identifies
the table that is current. The steps that follow generate the first batch of Independent Ancestor Sets. 
Note that initially, many Ancestor Sets will have the same entries.
\begin{enumerate}
\item Since all Independent Ancestor Sets are Prime Ancestor Sets, build
the Prime Ancestor Sets for $triples$.\footnote{The algorithm in Appendix \ref{fastInd}
builds the first batch of independent sets directly and quickly. However, the steps listed here 
are easy to follow, and prove that the iteration will succeed.}
\item Order the list of Prime Ancestor Sets  by increasing size. 
\item Discard all Ancestor Sets $Anc(P)$ 
such that $P\notin Anc(P)$. ($P$ will belong to $Surplus$.) 
\item The smallest Ancestor Set, $Anc(Q)$, must be Independent. Move it to the list of Independents.
\item Every Ancestor Set $B$ such that $Q\in B$ equals or contains $Anc(Q)$.  Discard these sets.
\item If any Ancestor Sets remain, iterate from (4).
\end{enumerate}

In the two sections that follow, we assign table $triples(A_i)$ to each
Independent Ancestor Set $A_i$ and then reduce the size of $triples$ 
in preparation for finding the next batch of Independent Ancestor Sets. 
Every Independent Ancestor Set that has been discovered
causes a piece of the current $triples$ table to be removed  
and additional covered indices in $triples$ to be set to 0.

\subsection{Assigning $triples(A_k)$ to Independent $A_k$}
Suppose that $A_k$ is Independent. Recall that $triples(A_k)$ is the set of all triples $(r,i,j)$ such that $r$, $i$, and $j$, if $j>0$, belong to $A_k$. Table $triples(A_k)$ contains all covering relationships for each $r\in A_k$. Given a positive additive cost function, this information determines a minimum-cost set of primes that covers $PS_k$ (and $A_k$) and is part of a minimum-cost basis for $f$. 

Since the Independent Ancestor Sets $\{A_k\}$ are disjoint, the arrays $triples(A_k)$ also are disjoint. The $triples(A_k)$ tables are called the \emph{Independent triples tables}.  

All of the entries in these tables now can be removed from the main $triples$ table.

\subsection{Cascading the Union of the Independent Ancestor Sets}\label{cascade2}
As was noted in Section \ref{ancThm}, the Ancestor Theorem shows that the minimum-cost covers of the current set of Independent Ancestor Sets are disjoint pieces of the overall minimum-cost basis. 

A minimum-cost cover of $PS$ depends only on
$triples(A)$ and the given cost function $C$. 
From the point of view of the other Independent Ancestor Sets, 
this separate piece of the problem can be viewed as solved. 
All of the indices covered  by $PS$ can be removed from the problem 
by cascading $A$ through $triples$. 

We already have removed each $triples(A)$ from the table but other downstream indices 
may be covered by the indices in $A$. Note that the individual Independent Ancestor Sets do not have to be cascaded separately.
If $A_1,\dots ,A_m$ form the current batch of Independent Ancestor Sets, let $AU$ denote their union.
The covered indices are set to $0$ by performing a removal cascade of $AU$ through the current $triples$. That is, set the initial cascade list to $AU$ and,
\begin{enumerate}
\item Replace each index belonging to the cascade list with $0$ in the current $triples$.
\item An $(s,0,0)$ entry indicates that $s$ has been covered by $AU$. Remove all $(s,x,y)$ entries 
and append index $s$ to the \emph{newIndex} cascade list.
\item Iterate, setting all new $s$ indices to $0$ in the triples table, until no new covered indices are found. 
\end{enumerate}

The final cascade list equals $AU+extra~indices.$ 
The extra indices that correspond to primes belong to $Surplus$.

As a result of the cascade, all indices that were in $AU$ and any extra indices that were covered by $AU$ have been removed from the current $triples$ 
and hence have been removed from the problem. 
Specifically, if any $X_t$ in an Independent $A$ belongs to non-Independent $B$, then, since $Anc(X_t)=A$,  all of $A$ is contained in $B$.  A cascade of $A$ through $triples(B)$ 
would remove all indices in $A$ from $B$ and also would remove any extra indices in $B$ that are covered by $A$.
 
\subsection{The Iterative Ancestor Set Algorithm}\label{BigIter}
 At this point, the initial Independent Ancestor Sets $\{A_k\}$ and tables $triples(A_k)$ have been set aside and the current $triples$ has been reduced and simplified. If the current table is not empty, iterating the preceding steps for the current table will reveal new disjoint Independent Ancestor Sets with their own associated tables. 
 Since the size of the current $triples$ table decreases at each iteration, the process will terminate. 

The overall process is summarized below.
The following operations build the list of Independent Ancestor Sets and their disjoint primesets.
\begin{enumerate}
\item Remove duplicate entries and useless entries 
(i.e.  $(r,i,j)$ when there is an entry $(r,i,0)$ or $(r,j,0)$) from the current $triples$. Remove inactive indices. 
\item Find the Prime Ancestor Sets for the current $triples$.
\item Identify a new batch of Independent Ancestor Sets.
\item For each Independent Ancestor Set $A_k$, let $PS_k$ be its primeset
and find $triples(A_k)$. 
Associate $triples(A_k)$ with $PS_k$ and 
remove the $triples(A_k)$ entries  from the current $triples$ table. 
\item Cascade the union of the new Independent Ancestor Sets 
through the current $triples$ table in order to remove 
all other products covered by the current batch of Independent
Ancestor Sets from the problem.
\item If the resulting table is not empty, iterate from step 1.
\end{enumerate}

Since the smallest Ancestor Set in a batch always is Independent, at each iteration, the total number of 
Independent Ancestor Sets will increase and the number of entries in the 
current $triples$ will decrease and eventually become zero.  

At that point, the set of Independent Ancestor Sets $\{A_i\}$ is 
complete and all of the disjoint primesets $PS_i$ have been identified.

\section{Summary of the Theorems}
We summarize the results of the preceding sections in the following theorems.

\begin{theorem}[Cascade List] 
The cascade list resulting from cascading the union of the canonical Independent primesets $PS_i$ of Boolean function $f$ through $triples(f)$ contains all of the products in $triples(f)$.
\end{theorem}
\proof
Each $PS_i$ covers $A_i$. The reduction process cascades the union of the Independent Ancestor Sets $A_i$ through the table, covering all of the active, non-free products. Since  the active products cover all of  the inactive non-free products all non-free products will be covered.
\endproof

\begin{theorem} \label{surplusProof}
Given a Boolean function $f$ with Independent Ancestor Sets $A_i$ and primesets $PS_i\subset A_i$, 
let $Surplus$ be the set of non-free primes that do \emph{not} belong to $\bigcup{PS_i}$.
Then no prime in $Surplus$ can be part of \emph{any} minimum-cost basis.
\end{theorem}
\proof
By definition, $P\in Surplus$ does not belong to any $PS_i$ and hence is not an ancestor of any prime that belongs to any $PS_i$.  
The products in $A_i$ can only be covered by a subset $QS_i$ of $PS_i$ that covers $PS_i$, and the union of the $QS_i$ covers \emph{all} non-free products, including the products in $Surplus$.  Therefore, any prime $P\in Surplus$ that was part of a basis would be covered by other primes
that are in the basis, so the basis would not have minimum cost.
\endproof

\begin{theorem}[Canonical Partition]
Let $f$ be a Boolean function. There is a canonical partition of the 
complete set of primes of $f$ into disjoint sets of primes, 
\[(Essential~Primes) \cup (PS_1) \cup \dots  \cup (PS_N) \cup (Unnecessary~Primes)\]
where
\begin{enumerate}
\item The Essential Primes must be part of any basis for $f$.
\item $PS_1, \dots,PS_N$ are the 
primesets of the Independent Ancestor Sets $A_1,\dots,A_N$ for $f$.
\item The Independent Ancestor Sets for $f$ are disjoint.
\item The Unnecessary Primes consist of the union of the free non-essential primes
with $Surplus$, and an Unnecessary Prime
cannot be part of a minimum-cost basis for any positive,
additive cost function.
\item For each $PS_i$, there is a canonical triples table $TS_i$ 
that contains all covering relationships between
the products that belong to $A_i$.
\item Given a positive additive cost function $C$, any minimum-cost 
basis for $f$ with respect to $C$ must consist of the union of the 
Essential Primes with subsets $QS_i(C)\subset PS_i$ such that
$QS_i(C)$ is a minimum-cost cover of $PS_i$. (That is, 
$Cascade(QS_i(C),TS_i)$ is empty.)
\end{enumerate}
\end{theorem}
§
The sections that follow deal with the spans of primes, which are the keys to the solutions $QS_i$ of some uncomplicated Independent Ancestor Sets $A_i$. 

\section{Spans of Individual Primes and Simple Solutions\label{primeSpans} }
Suppose that $A$ is an Independent Ancestor Set for Boolean function $f$ and $PS$ 
is the primeset for $A$. The set of spans of the individual primes $P\in PS$ are of special interest.  
Recall that the span of a prime $P$ in $A$ equals the products in $A$ that are covered by $P$. 

\begin{lemma} Prime $P$ is an ancestor of every product in $Span( P)$.
\end{lemma}
\begin{proof}
Clear from the definition of the span.
\end{proof} 

\begin{definition}\label{maxSpan} The span of prime $P_r\in A$ is \emph{maximal} if the span includes at least one prime other than $P_r$ and $Span(P_r)$ is not properly contained in the span of any other prime in $A$.
\end{definition}

\begin{definition}\label{closedSpan} The span of prime $P_r\in A$ is \emph{closed} if 
$triples(A)$ includes a triple $(r,i,j)$ such that $X_i \in Span(P_r)$ and, if $j>0)$, 
$X_j \in Span(P_r)$.
\end{definition}

\begin{lemma}\label{spanSubset}  If $Q \in Span( P)$ then $Span( Q) \subset Span( P)$.
\end{lemma}
\begin{proof}
Clear from the definition of a span.
\end{proof}

\begin{definition}
Primes $P$ and $Q$ are equivalent if and only if  $P$ is in the span of $Q$ and $Q$  is in the span of $P$.
\end{definition}

\begin{lemma}[Equivalence Lemma] If $P$ and $Q$ are equivalent, then  $Span( P)=Span( Q)$ and the span is closed with respect to both $P$ and $Q$.
\end{lemma}
\begin{proof} 
By Lemma (\ref{spanSubset}), \[Span( Q) \subset Span( P)\]
\[Span( P) \subset Span( Q).\]
\end{proof}
For equivalent primes $P$ and $Q$, whenever $P$ is covered, $Q$ also is covered and vice-versa. A set of primes that are equivalent can be represented by the one with the lowest cost.

\begin{lemma} If $P$ is equivalent to $Q$, then 
\begin{itemize}
\item[] $Anc(P)=Anc(Q)$.
\item[] $Anc(P)$ and $Anc(Q)$ are closed with respect to $P$ and $Q$.
\end{itemize}
\end{lemma}
\proof
 $Anc(P)=Anc(Q)$ since $P$ is an ancestor of $Q$ and $Q$ is an ancestor of $P$, so
 $Anc(P)\subset Anc(Q)$ and $Anc(Q)\subset Anc(P)$.
 $P\in Anc(P)$ since $P\in Anc(Q)=Anc(P)$. Similarly, $Q\in Anc(Q)$.
 \endproof

\subsection{The Span Basis Theorem}\label{spanBasis}
It is not unusual for an Independent Ancestor Set to be equal to the span of one of its primes.
 The Span Basis Lemma states that in this case a minimum-cost cover for $A$ equals a minimum-cost cover 
 for $P$.

\begin{lemma}[Span Basis Lemma]
Suppose that $A$=$Anc(P)$ is an Independent Ancestor Set and $P\in~A$.
If $Span(P)$ in $triples(A)$ equals $A$, then $Span(P)$ is maximal and closed, and $Q_1,\dots Q_t$ is a 
minimum-cost cover for $A$ if and only if it is a minimum-cost cover for $P$.
\end{lemma}
\proof
It is clear that $Span(P)$ is maximal and closed. 
Any cover for $A$ covers $P$, and any cover for $P$ covers $A$.
Hence, both have the same minimum-cost covers.
\endproof

The most desirable result of the Span Basis Lemma would be that a single prime $P$ is the minimum-cost cover for $A$.  The Span Basis Theorem that follows shows that this is the case 
for the most common cost functions when $P$ is equal to the minimum-cost prime in its equivalence class.  In this case, $P$ must belong to any minimum-cost basis for $f$. 

\begin{definition}
Cost function $C$ is a \emph{constant additive} cost function if $C(X)$ 
equals a constant $c$ for each product $X$, and $C$ is additive.
\end{definition}

\begin{theorem}[Span Basis Theorem]\label{SpanBasisTheorem}
Suppose that $f(x_1,\dots, x_n)$ is a Boolean function, and
\begin{enumerate}
\item $A$=$Anc(P)$ is an Independent Ancestor Set and $P\in~A$,
\item $Span(P)$ in $triples(A)=A$, 
\item for every prime $Q$ equivalent to $P$, $C(P)<=C(Q)$,\label{QequivP}
\item the cost function is either constant additive or satisfies Quine's hypothesis (namely that
the addition of a single factor to a formula increases the cost of the formula).
\end{enumerate}
Then $P$ is a minimum-cost cover for $A$, and $P$ 
(or an equivalent prime with the same cost) must be part of any minimum-cost basis. 
\end{theorem}
\proof
Condition (\ref{QequivP}) can be guaranteed when the non-free primes are ordered by cost, 
and the span of the least-cost prime in a set of equivalent ones is always selected 
to represent the set of equivalent spans. In this case, any minimum cost cover of $P$
other than $P$ would consist of two or more primes. 
If $C$ is constant then the cost of any cover of $P$ 
with two or more primes is $>C(P)$. 
If $C$ satisfies Quine's hypothesis, then for any cover consisting of two or more 
primes, the union of the factors 
of the primes in such a cover of $P$ would have to properly contain every factor of $P$
along with other factors. Thus if the cost function 
satisfies Quine's hypothesis,  such a cover could not be minimum-cost.
\endproof

\subsection{Splitting an Independent Ancestor Set}
It may appear that it is not possible to further collapse an Independent Ancestor Set, but in fact,  an Independent Ancestor Set can be decomposed, (simplifying its solution) if it contains an \emph{Independent prime}.

\begin{definition} Suppose that prime $P$ belongs to Independent Ancestor Set $A$.
We say that prime $P$ is \emph{Independent} if $Span(P)$ in $triples(A)$ is maximal and closed, $Span(P)\neq A$, and 
$P$ does not belong to the span of its complement in $A$. 
\end{definition}
In this case, any minimum-cost cover for $A$ consists of a minimum-cost cover for $P$ (and 
hence, for $Span($P)) plus a minimum-cost cover of the primes in $A$ and outside of $Span(P)$.

The independence of $P$ can be checked by computing the span in $triples(A)$ of the set of products in $A$ that do not belong to $Span(P)$. If $P$ is not covered, then $P$ is Independent. 

A good candidate for independence is the prime in $A$ with the largest span, such that $Span(P)$ is  maximal and closed, and such that for every entry corresponding to $(P, X_i, X_j)$ in the triples 
table, $Xi$, and $Xj,~ (if ~X_j > 0)$, belong to $Span(P)$.  
The following Theorem describes the partition in more detail.
 
\begin{theorem}[Independent Prime Decomposition Theorem]\label{IndPrimeDecomp}   
If $P \in A$ is Independent, then $A$ can be split into $Span(P$) and one or more smaller 
Ancestor Sets.
\end{theorem}
\proof
The splitting algorithm is:
\begin{enumerate}
\item Extract from $triples(A)$:\\
$triples(Span(P))= \{(r, i, j)\}~\text{such that } r, i , \text{ and } j (\text{ if } j>0)) \text{ belong to~} Span(P).$
\item Cascade Span(P) through the remaining entries.
\item	Calculate the Independent Ancestor Sets for the remaining triples.
\end{enumerate}
Thus, the solution to $A$ consists of a minimum-cost cover for $Span(P)$, as in Section(\ref{spanBasis}) plus the solutions to the second-generation Independent Ancestor Sets associated with the reduced table.
\endproof

If $C$ is constant or satisfies Quine's hypothesis and $C(P)<=C(Q)$ 
for every prime $Q$ equivalent to Independent prime $P$, 
then $P\in Anc(P)$ is a minimum-cost cover for $Span(P)$ and $P$ (or an equivalent prime with the same cost) must be part of any minimum-cost basis. 

Note that one or more of the second-generation Independent Ancestor Sets might contain Independent primes, and the process could iterate several times. However, the size of the new Ancestor Sets and of the their tables decreases at each iteration and the number of steps is linear in the data.

\appendix

\section{An Algorithm for Reordering the Non-free Products in ALL}\label{reorder}
The following algorithm is convenient for problems of moderate size. 
The algorithm reorders $ALL$ while also computing 
the span of each non-free prime $P_k,~k=1\dots M$, 
as specified in definition \ref{spanDef}, that is,
\[Span(P_k)=Consensus(<Free>, P_k).\] 
The reordering places all non-free products that are in the span of a non-free prime
before any products that are not covered by an individual prime.
Recall that $Span(P_k)$ is closed if there are non-free products $X_i, X_j\in Span(P_k)$ with
$P_k=(X_i \circ X_j)$.

The algorithm also generates some special entries for the $triples0$ table. 
Since by definition, $P_k$ covers every $X_s$
in its span, each triple $(s,k,0)$ is a valid entry for $triples0$.
However, the algorithm below only saves entries that arise from maximal closed spans. 

\begin{enumerate}
\item Initialize $newALL=$Essentials, non-Essential frees, non-free~ primes.
\item For each $P_k$,
\begin{enumerate}
\item compute $Span(P_k)$.
\item for each non-free $Y\in Span(P_k)$ that is not in $newAll$,
append $Y$ to $newALL$. 
\end{enumerate}
\item For each $Z\in ALL$ that is not in $newALL$, append $Z$ to $newAll$.
\item Sort the spans by decreasing size and identify the maximal closed spans,
\[Span(Q_1),\dots Span(Q_h).\]
\item For each $X_r\in Span(Q_i)$, append $(r,i,0)$ to $triples0$.
\end{enumerate}

\section{An Algorithm for Generating the $triples$ Table}\label{genTable}
The first set of entries for $triples0$ were generated in Appendix \ref{reorder},
and consist of $(r,i,0)$ such that $X_r$ belongs to a maximal closed $Span(P_i)$.

The algorithm that follows  generates ordinary $triples0$ entries  
by taking the loose Iterated Consensus of non-free, non-prime products $X_i$ with free products.

Then the algorithm calculates $triples3$ entries by taking the consensus 
of pairs of non-free products $X_i$ and $X_j$.
For each non-free result $X_r$,  $(r,i,j)$ is tested to see if it is useful.
Recall that $(r,i,j)$ is useful if there is no $(r,i,0)$ or $(r,j,0)$ in $triples0$.
Only useful entries are appended to $triples3$.
 
 It is not difficult to organize the data in ways that speed up the processing.

\begin{enumerate}
\item Set $productList$ equal to the non-free products in $newALL$.
\item Part 1. Calculate ordinary $triples0$ entries:
\begin{enumerate}
\item For  each non-prime $X_i$ on the productList, compute the consensus of $X_i$ with each free product. If the result  $X_r=(X_i~o~ free)$ is non-free and $(r,i,0)$ is not a duplicate, 
append it to $triples0$.
\end{enumerate}
\item Part 2. Calculate the $triples3$ entries:
\begin{enumerate}
\item Perform Loose Iterated Consensus on the entries in productList. 
\item Test each non-free result $X_r=(X_i \circ X_j)$. If it is useful,  
append $(r,i,j)$ to $triples3$.
\end{enumerate}
\end{enumerate}

Note that there has to be at least one table entry $(r,i,j)$ for each non-free product $X_r$. 
After removing all inactive entries, there are entries for each active
non-free product. It is convenient to sort the table by $r-values$.

\section{A Direct Algorithm for Independent Ancestor Sets}\label{fastInd}
We assume that the non-free primes $P_1,\dots P_n$ have been arranged 
in order of increasing cost.  We compute the ancestor sets of the primes
starting from the most costly.

If $P_i$,  $i<n$, is a parent of $P_n$, 
then $Anc(P_i) \subseteq Anc(P_n)$.  
If the sets are equal, the least costly prime $P_i$ will be preferred to $P_n$ 
as the representative prime for the set.
If $Anc(P_i)$ is a proper subset, $Anc(P_n)$ cannot be independent. 
In either case, $P_n$ can be discarded.

If $P_i$,  $i<n$, is a grandparent of $P_n$, then again  $Anc(P_i) \subseteq Anc(P_n)$ 
and $P_n$ can be discarded. More generally, checking each generation of 
the ancestors of the remaining primes (processed in order of decreasing cost)
leads to more discards.  
The number of primes whose ancestor sets require full expansion shrinks quickly. 

One last check of the remaining set of primes is needed.
If  $P_j \in Anc(P_i)$ for $j>i$, then  $Anc(Pj)$ must be a proper subset of $Anc(P_i)$;
otherwise, $P_i \in Anc(P_j)$ and $P_j$ would have been discarded. 
Thus, $P_i$ can be discarded.
When this test has been completed, the remaining primes and their ancestor sets form the first batch of  Independent ancestor sets.

After extracting each set $triples(Anc(P_k))$ and cascading 
the union of the ancestor sets in this batch through
the triples table, the process is iterated until triples is empty.

Once all of the independent sets and their associated triples have been found, 
each Independent set can be checked to find the prime in the set whose span (within its triples set) is largest.
If  $Q \in Anc(P)$ and $Span(Q)=Anc(Q)$ (and hence equals $Anc(P)$) and the cost function is 
constant additive or satisfies Quine's hypothesis, 
then $Anc(Q)$ is solved. If $Span(Q)$ is not equal to $Anc(Q)$
but $Span(Q)$ is maximal and closed,
then $Q$ can be tested to see whether it is independent, which would enable the ancestor
set to be split. 



\end{document}